%% file: growth.tex
\documentclass[a4paper]{amsart}

\usepackage{amsmath}
\usepackage{amsthm}
\usepackage{amssymb}
\usepackage{amsfonts}
\usepackage{stmaryrd}
\usepackage{url}
\numberwithin{equation}{section}
\numberwithin{figure}{section}
\numberwithin{table}{section}
\usepackage{refcount}
\usepackage{enumitem}
\usepackage[all]{xy}
\usepackage{enumitem}

\newenvironment{renum}{
  \begin{enumerate}[label=\textup{(\roman*)}]
}{
  \end{enumerate}
}
\usepackage{hyperref}
\theoremstyle{plain}
\newtheorem{thm}[subsection]{Theorem}
\newtheorem*{thm*}{Theorem}
\newtheorem{cor}[subsection]{Corollary}

\newtheorem{lem}[subsection]{Lemma}
\newtheorem{pro}[subsection]{Proposition}

\newtheorem{thmx}{Theorem}

\newtheorem{conjx}[thmx]{Conjecture}

\theoremstyle{definition}
\newtheorem{exa}[subsection]{Example}
\newtheorem{dfn}[subsection]{Definition}

\theoremstyle{remark}
\newtheorem{remark}[subsection]{Remark}

\newcommand\no{n\textsuperscript{o}}   

\newcommand{\bb}[1]{\mathbb{#1}}
\newcommand\NN{\bb{N}}
\newcommand\ZZ{\bb{Z}}
\newcommand\QQ{\bb{Q}}
\newcommand\CC{\bb{C}}
\newcommand\RR{\bb{R}}
\newcommand\HH{\bb{H}}

\DeclareMathAlphabet{\mathitbf}{OML}{cmm}{b}{it}

\newcommand{\bo}[1]{\mathitbf{#1}}

\newcommand{\ca}[1]{\mathcal{#1}}

\DeclareMathAlphabet{\mathpzc}{OT1}{pzc}{m}{it}

\let\epsilon\varepsilon
\let\phi\varphi
\let\theta\vartheta
\newcommand\eps\epsilon

\newcommand{\argu}{\hbox to 1.5ex{\hrulefill}}
\newcommand{\image}{\operatorname{im}}
\newcommand{\dbl}{\llbracket}
\newcommand{\dbr}{\rrbracket}

\newcommand{\CX}{\ensuremath{\mathpzc{X}}}
\newcommand{\CM}{\ensuremath{\mathpzc{M}}}

\newcommand{\iden}{\operatorname{id}}

\newcommand{\GL}{\operatorname{GL}}
\DeclareMathOperator{\Cay}{Cay}

\usepackage[all]{xy}
\usepackage{longtable}
\input{graphs-defs}

\usepackage{geometry}

\title{On the growth of a Coxeter group}
\author{T. Terragni} 

\date{\today}

\usepackage{tikz}
\usetikzlibrary{matrix,arrows}
\DeclareMathOperator{\ord}{ord}
\DeclareMathOperator{\cleq}{\preceq}
\newcommand{\WS}{(W,S)}
\newcommand{\WSp}{(W',S')}
\newcommand{\pWS}{p_{\WS}(t)}
\newcommand{\WI}{(W_I,I)}
\newcommand{\pWI}{p_{\WI}(t)}

\newcommand{\spr}{\ca{F}}

\newcommand{\caR}{\ca{R}}

\newcommand{\bow}{\bo{w}}
\newcommand{\bov}{\bo{v}}
\newcommand{\bou}{\bo{u}}

\DeclareMathOperator{\PGL}{PGL}

\newcommand{\so}{s_0}
\newcommand{\ro}{r_0}

\begin{document}
\maketitle
\footnotetext{{\it Email:} \href{mailto:tom.terragni@gmail.com}{\texttt{tom.terragni@gmail.com}}{\it, Web:} \url{https://sites.google.com/site/tomterragni}.}

\begin{abstract}
For a Coxeter system $\WS$ let $a_n^{\WS}$ be the cardinality of the sphere of radius $n$ in the Cayley graph of $W$ with respect to the standard generating set $S$. 
It is shown that, if $\WS\cleq\WSp$ then $a_n^{\WS}\leq a_n^{\WSp}$ for all $n\in \NN_0$, where $\cleq$ is a suitable partial order on Coxeter systems (cf. Thm.~\ref{thm:an}). 

It is proven that there exists a constant $\tau= 1.13\dots$ such that for any non-affine, non-spherical Coxeter system $\WS$ the growth rate $\omega\WS=\limsup \sqrt[n]{a_n}$ satisfies $\omega\WS\geq \tau$ (cf.~Thm.~\ref{thm:tau}). 
The constant $\tau$ is a Perron number of degree $127$ over $\QQ$.

For a Coxeter group $W$ the Coxeter generating set is not unique (up to $W$-conjugacy), but there is a standard procedure, the \emph{diagram twisting} 
(cf.~\cite{brady-etal--rcgag}), which allows one to pass from one Coxeter generating set $S$ to another Coxeter generating set $\mu(S)$.
A generalisation of the diagram twisting is introduced, the \emph{mutation}, and it is proven that Poincar\'e series are invariant under mutations (cf.~Thm.~\ref{thm:mu}).

\medskip
\noindent\textbf{2010 MSC:} 20F55 (Primary), 20F32, 05C25 (Secondary).

\smallskip
\noindent\textbf{Keywords:} Coxeter groups, growth of groups.
\end{abstract}

\section*{Introduction}
The growth of finitely generated groups has been the subject of intensive investigations (cf.~\cite{grigorchuk--bppg,grigorchuk--dgfggtim}, \cite{grigorchuk-delaharpe--prgesgt}, \cite{delaharpe--tggt}) and led to ground-breaking results, e.g., M.~Gromov showed that a finitely generated group has polynomial growth if, and only if, it is virtually nilpotent (cf.~\cite{gromov--gpgem}).

For a group $G$ being generated by a finite symmetric set $X\subseteq G$ not containing the identity $1\in G$, the growth rate\footnote{The growth rate is often called \emph{exponential growth rate}.} is defined by $\omega(G,X)=\limsup_n \sqrt[n]{a_n}$, where $a_n$ is the number of elements in $G$ which can be written as a product of $n$ elements in $X$ but which cannot be written as a product of less than $n$ elements in $X$. 
If $G$ is of subexponential growth, i.e., polynomial or intermediate growth, then $\omega(G,X)\leq 1$. 

The set of isomorphism classes of Coxeter systems admits a partial order $\cleq$, and the corresponding monotonicity result for growth sequences is proven.
\begin{thm*}[\ref{thm:an}]
Let $\WS$ and $\WSp$ be Coxeter systems. If $\WS\cleq \WSp$ then $a_n^{\WS}\leq a_n^{\WSp}$ for all $n\in \NN_0$.
\end{thm*}

Spherical and affine Coxeter systems have, respectively, growth rate zero and one. One of the main results of this paper can be stated as follows.
\begin{thm*}[\ref{thm:tau}] 
Let $\WS$ be a non-affine, non-spherical Coxeter system. 
Then its growth rate satisfies $\omega\WS\geq \tau$, where $\tau = 1.13\dots$ is an algebraic integer of degree $127$ over $\QQ$, which is also a Perron number with minimal polynomial $m_\tau(t)$ given in \S\ref{s:tau}. 
Moreover, $\tau=\omega\WS$, where $\WS$ is the hyperbolic Coxeter system $E_{10}$. 
\end{thm*}
A remarkable coincidence occurrs (cf.~Rem.~\ref{rem:notes-tau}). 
Besides having the smallest minimial growth rate among Coxeter systems, $E_{10}$ is also known to minimise a certain function $\lambda_\rho$, which reflects, in the hyperbolic case, the metric properties of the orbifold defined by Tits' representation $\rho$ (cf.~\cite{mcmullen--cgsnhm}). 

For a group $G$ with a finite symmetric generating set $X\subseteq G\setminus \{1\}$ one defines the growth series by $p_{(G,X)}(t)=\sum_{n\in \NN_0} a_n t^n \in \ZZ\dbl t\dbr$, thus $\omega(G,X)$ coincides with the inverse of the radius of convergence of $p_{(G,X)}(t)$, considered as a power series over $\CC$. 
For a Coxeter system $\WS$ the growth series is also called the \emph{Poincar\'e series} of $\WS$.

In \S\ref{s:rig-growth} we define the new notion of a \emph{mutation} $\mu(M,X,Y,\sigma)$ of a Coxeter matrix $M$, which induces an equivalence relation $\sim$ on Coxeter systems. 
Mutations generalise diagram twisting (cf.~\cite{brady-etal--rcgag}), but in general they do not preserve the isomorphism class of the group. 
Nevertheless, the Poincar\'e series is invariant under mutations of the Coxeter matrix.
\begin{thm*}[\ref{thm:mu}]Let $\WS$ and $(W',S')$ be Coxeter systems satisfying $\WS\sim (W',S')$. Then $p_{\WS}(t)=p_{(W',S')}(t)$.
\end{thm*}
Thus, mutations provide a tool to produce finitely many non-isomorphic Coxeter groups with the same growth series. It is an open problem whether there exist infinitely many groups with the same growth series (cf.~\cite[Ch. 1, Pbl.s 1--2]{mann--hgg}).

\subsection*{Acknowledgements} 
I wish to thank J.~Parkinson, P.~Spiga and Th.~Weigel for helpful discussions, and M.~Bucher, T.~Smirnova-Nagnibeda and A.~Talambutsa for the simulating conversations during the conference ``Geometric and analytic group theory'' in Ventotene, Italy.

\noindent{}I would also like to thank R. Howlett for providing a copy of \cite{howlett-muehlherr--icgwdnpr}, and the anonymous referees for helpful comments.

\section{Growth of finitely generated groups}\label{s:growth}
Let $G$ be a finitely generated group, and let $X=X^{-1}\subseteq G\setminus\{1\}$ be a finite, symmetric set of generators. 
The \emph{length} of $g\in G$ with respect to $X$ is the minimal $n$ such that $g=x_1x_2\dots x_n$ with $x_i\in X$; the \emph{length function} will be denoted by  $\ell_{(G,X)}\colon G\to \NN_0$. It has a natural interpretation in terms of the metric on the Cayley graph $\Cay(G,X)$.

For $n\in \NN_0$, the ball in $\Cay(G,X)$ centred around $1_G$ with radius $n$ will be denoted by $B^{(G,X)}_n=\{g\in G\mid \ell_{(G,X)}(g)\leq n\}$, the corresponding sphere by $A^{(G,X)}_n=\{g\in G\mid \ell_{(G,X)}(g)= n\}$. Their sizes are $a^{(G,X)}_n=|A^{(G,X)}_n|$ and $b^{(G,X)}_n=|B^{(G,X)}_n|$.

The central objects under investigation are the growth series 
\[p_{(G,X)}(t)=\sum_{n\in \NN_0} a_n^{(G,X)} t^n\in \ZZ\dbl t\dbr,\]
and the growth rate $\omega(G,X) =\limsup_{n\to\infty}\sqrt[n]{a_n^{(G,X)}}$. 

Note that $G$ has \emph{exponential growth} if $\omega(G,X)>1$ for some (and hence any) generating set $X$. 
The present paper only deals with finitely generated linear groups $G$. Therefore, $G$ has \emph{polynomial growth} with respect to some (and hence any) generating system $X$ if $\omega(G,X)\leq 1$ (cf.~\cite[Cor. 5]{tits--fslg}).

The \emph{minimal growth rate} $\omega(G)$ is the infimum of $\omega(G,X)$, as $X$ runs over all finite, symmetric generating sets of $G$.

\section{Coxeter groups}\label{s:cox-gps}
Standard references for Coxeter groups include  \cite{bourbaki--gal46,humphreys--rgcg}.
\subsection{Coxeter systems}\label{ss:cox-syst}
Let $S$ be a finite set, and let $M$ be an $(S\times S)$-matrix such that 
$m_{s,s}=1$, and $m_{s,r}=m_{r,s}\in \ZZ_{\geq 2}\cup\{\infty\}$ for all $s,r\in S$, $s\neq r$. Then $M$ is a \emph{Coxeter matrix} over $S$. 

The \emph{Coxeter system} associated with a Coxeter matrix $M$ over $S$ is the pair $\WS$ where 
\begin{equation}\label{eq:cox-pres}
W=W(M)=\langle S\mid (sr)^{m_{s,r}} \,\text{ if }m_{s,r}<\infty\rangle.
\end{equation}

The Coxeter matrix $M$ (or, equivalently, the presentation \eqref{eq:cox-pres}) is often encoded in the Coxeter graph $\Gamma(M)$ (cf.~\cite[Ch.~IV \no 1.9]{bourbaki--gal46}).
Either datum is called the \emph{type} of $\WS$.

If $I\subseteq S$ let $W_I=\langle I\rangle \leq W$. The \emph{parabolic subsystem} $(W_I,I)$ is a Coxeter system in its own right, with Coxeter matrix $M_I=(m_{s,r})_{s,r\in I}$.
Its Coxeter graph is the graph induced from $\Gamma$ by the vertices in $I$, and 
\begin{equation}\label{eq:l}
\ell_{(W_I,I)}=\ell_{\WS}|_{W_I}.
\end{equation}
The finite set $\spr=\spr(W,S)=\{I\subseteq S\mid |W_I|<\infty\}$ is called the set of \emph{spherical residues}.

A \emph{Coxeter-isomorphism} $\phi\colon \WS\to (W',S')$ of Coxeter systems of types $M$ and $M'$ respectively, is a bijection $\phi\colon S\to S'$ such that $m'_{\phi(s),\phi(r)}=m_{s,r}$ for all $s,r\in S$.

Any Coxeter group $\WS$ is linear via the Tits' reflection representation $\rho\colon W\to \GL(\RR^S)$ (cf.~\cite[Ch.~V, \S4]{bourbaki--gal46}). The representation $\rho$ is determined by the symmetric matrix\footnote{For short, we put $\frac\pi\infty=0$.} $B=B_M=\left(-\cos\frac{\pi}{m_{s,r}}\right)_{s,r\in S}$, and the signature of $B$ induces the following tetrachotomy on irreducible Coxeter systems.

\begin{renum}
\item If $B$ is positive definite, then $\WS$ is \emph{spherical},
\item if $B$ is positive semidefinite with $0$ a simple eigenvalue, then $\WS$ is \emph{affine}, 
\item if $B$ has $|S|-1$ positive and $1$ negative eigenvalue, then $\WS$ is \emph{hyperbolic}\footnote{There are several non-compatible notions of hyperbolicty, cf.~\cite[Note 6.9]{davis--gtcg}. In the present work ``hyperbolic'' coincides with Bourbaki's notion (cf.~\cite[Ch. V, \S4, Ex.13]{bourbaki--gal46}).}, or
\item none of the above conditions applies.
\end{renum}
The irreducible Coxeter system $\WS$ is spherical if, and only if, $W$ is a finite group. The classification of spherical and affine systems is classical (cf.~\cite[Ch. VI]{bourbaki--gal46}). 
For a characterisation of hyperbolic Coxeter systems see \S\ref{ss:min}.

\subsection{The word problem}\label{ss:word-pbl}
If $S$ is a finite set, let $S^\ast$ be the free monoid\footnote{Words in $S^\ast$ are denoted in boldface: $\bow =s_1 s_2 \dots s_n\in S^\ast$. A subword $\bow'$ of $\bow$ is either the empty word $\boldsymbol{1}$ or a word of the form $\bow'= s_i s_{i+1} \dots s_k$ for  $1\leq i\leq k\leq n$.} over $S$, equipped with the natural $\NN_0$-grading $\deg\colon S^\ast\mapsto \NN_0$, $\deg(s)=1$ for all $s\in S$, and the ShortLex total order with respect to some total order on $S$ (cf.~\cite[\S2.5]{epstein-etal--wpg}).
For $s,t\in S$ and $m\in \NN_0$ let  $[s,t,m]\in S^\ast$ be the word
\[ [s,t,m]=\begin{cases}
(st)^{m/2}&\text{ if }2\mid m,\\
(st)^{\frac{m-1}{2}}s&\text{ if }2\nmid m.
\end{cases}
\]

Let $M$ be a Coxeter matrix over $S$.
The \emph{$M$-operations} (or \emph{$M$-moves}) on $S^\ast$ are modifications of words of the following types:
\begin{equation}\label{eq:Mmoves}
\begin{array}{cl}
M^{(1)}\colon & \bov (ss)\bou \mapsto \bov\bou,\\
M^{(2)}\colon & \bov [s,r,m_{s,r}]\bou \mapsto \bov [r,s,m_{s,r}]\bou, \quad\text{ if }m_{s,r}<\infty.
\end{array}
\end{equation}

Let $(W,S)$ be the Coxeter system of type $M$, and let $\pi_M\colon S^\ast\to W(M)$ be the canonical projection (of monoids). Then, for all $\bow\in S^\ast$,
\begin{equation}\label{eq:length-deg}
\deg(\bow)\geq \ell_M(\pi_M(\bow)).
\end{equation}

A word $\bow\in S^\ast$ is called \emph{reduced} for $\WS$ if equality holds in \eqref{eq:length-deg}.
If $w\in W(M)$, there is a unique ShortLex-minimal element $\sigma_M(w)\in S^\ast$ such that $\pi_M\sigma_M(w)=w$. Thus, $\sigma_M\colon W(M)\to S^\ast$ is a section of $\pi_M$, with the additional property\footnote{Actually, any section of $\pi_M$ with property \eqref{eq:length-section} would suffice for the purposes of this paper.} that \begin{equation}\label{eq:length-section}
\deg(\sigma_M(w))=\ell_M(w).
\end{equation}  

A word $\bow\in S^\ast$ is called \emph{$M$-reduced} if its degree cannot be decreased by applying any finite sequence of $M$-operations. 
If two words $\bow,\bow'$ are connected by a sequence of $M$-moves, then they represent the same element in $W(M)$:
\begin{equation}\label{eq:piM}
\pi_M(\bow)=\pi_M(\bow'),
\end{equation} 
and hence reduced words are $M$-reduced.
Moreover, Tits solved the word problem as follows.

\begin{thm}[{\cite{tits--pmgc}, \cite[Ch.~IV, \S1, Ex.13]{bourbaki--gal46}}]\label{thm:word-pbl}
  Let $(W,S)$ be the Coxeter system with Coxeter matrix $M$. 
\begin{renum}
\item A word in $S^\ast$ is reduced for $\WS$ if, and only if, it is $M$-reduced.
\item\label{wp:li2} If $\bow,\bow'\in S^\ast$ are reduced words which represent the same element $\pi_M(\bow)=\pi_M(\bow')\in W$, then there is a sequence of $M$-operations taking $\bow$ to $\bow'$, and this sequence entirely consists of $M^{(2)}$-operations.
\end{renum}
\end{thm}

Following \cite[\S\S3.3--3.4]{bjorner-brenti--ccg}, let $\caR_M(w)=\{\bow\in S^\ast\mid \pi_M(\bow)=w,\text{ and }\deg(\bow)=\ell_M(w)\}$ be the set of the reduced words in $S^\ast$ representing $w\in W(M)$.
 
\begin{cor}\label{cor:misc-R} For $\bow,\bow'\in S^\ast$, and $w\in W(M)$ the following hold.
\begin{renum}
\item\label{li:sigma} $\sigma_M(w)\in \caR_M(w)$.
\item\label{li:closure} If $\bow\in \caR_M(w)$, and there exists a sequence $\bow \stackrel{M}{\longmapsto} \bow'$ of $M$-moves taking $\bow$ to $\bow'$, then $\bow'\in \caR_M(w)$.
\item\label{li:7} If $\pi_M(\bow)=\pi_M(\bow')$ and $\bow'$ is reduced, then there exists a sequence of $M$-moves 
\[\bow \longmapsto \bow'.\]
\end{renum}
\end{cor}

\subsection{Poincar\'e series}\label{s:poincare-ser}
Coxeter systems are pairs consisting of a finitely generated group $W$ and a finite, symmetric generating set $S$, and therefore the machinery described in \S\ref{s:growth} applies. 
In the context of Coxeter systems the growth series is also known as the \emph{Poincar\'e series} $\pWS$ of $\WS$. 
If $\WS$ is spherical then $\pWS$ is a polynomial, which can be explicitly computed in terms of the degrees of the polynomial invariants of $\WS$, simply known as the \emph{degrees} of $\WS$ (cf.~\cite{solomon--ofcg}, \cite[Ch.~3]{humphreys--rgcg}). 
For arbitrary Coxeter systems, the Poincar\'e series can be computed using the following property. 

\begin{pro}[{\cite{steinberg--elag}}]\label{pro:ps}
Let $\WS$ be a Coxeter system with Poincar\'e series $\pWS$. Then
\begin{equation}\label{eq:sph-steinberg}
  \frac{1}{p_{(W,S)}(t^{-1})}=\sum_{I\in \spr}\frac{(-1)^{|I|}}{\pWI},
\end{equation}
where $\spr=\spr\WS$. In particular, the Poincar\'e series $\pWS$ is a rational function.
\end{pro}

It is often possible to focus only on irreducible systems.
\begin{lem} Let $(W_1,S_1)$ and $(W_2,S_2)$ be Coxeter systems, and let $\WS=(W_1\times W_2,S_1\sqcup S_2)$ be their product. 
Then $\omega(W,S)=\max\{\omega{(W_1,S_1)}, \omega{(W_2,S_2)} \}$.
\end{lem}

\begin{proof}
The factorisation $p_{\WS}(t)= p_{(W_1,S_1)}(t)\cdot p_{(W_2,S_2)}(t)$ holds (cf.~\cite[Ch. IV, n.\!\textsuperscript{os}~1.8--1.9]{bourbaki--gal46}).
Since Poincar\'e series are series with non-negative coefficients and with degree-zero coefficient equal to one, then $\omega(W,S)\geq \max\{\omega(W_1,S_1), \omega(W_2,S_2)\}$. On the other hand, the product $p(t)$ of two rational functions $p_1(t)$ and $p_2(t)$ is holomorphic \emph{at least} in the smallest of the open disks centred in zero of radii $\rho_1, \rho_2$, where each of the two factors are holomorphic: thus
\[\omega\WS=\frac1\rho \leq \frac1{\min\{\rho_1,\rho_2\}}= \max\{\omega(W_1,S_1),\omega(W_2,S_2)\}.\qedhere\]
\end{proof}

\section{The partial order $\cleq$ on the class of Coxeter systems}\label{s:orders}
The core of the proof of Theorem~\ref{thm:tau} is the reduction to a finite set of elementary verifications. 
The tools which provide this reduction are the partial order $\cleq$ over the set of (Coxeter-isomorphism classes of) Coxeter systems, the corresponding monotonicity results, and the finiteness of the set of minimal non-affine, non-spherical Coxeter systems.

Let $\WS$ and $\WSp$ be Coxeter systems with Coxeter matrices $M$, $M'$ respectively. 
Define $\WS\cleq\WSp$ whenever there exists an injective map $\phi\colon S\to S'$ such that $m_{s,r}\leq m'_{\phi(s),\phi(r)}$ for all $s,r\in S$ (cf.~\cite[\S6]{mcmullen--cgsnhm}).

In particular, if $\WS$ and $\WSp$ are Coxeter-isomorphic (cf.~\S\ref{ss:cox-syst}) then $\WS\cleq\WSp$ and $\WSp\cleq\WS$. 
Therefore the preorder $\cleq$ descends to a partial order on the set of Coxeter-isomorphism classes of Coxeter systems. 
With a mild abuse of notation we will avoid the distinction between a Coxeter system and its Coxeter-isomorphism class.

\subsection{Monotonicity properties}

The partial order $\cleq$ has the following important property.

\begin{thmx}\label{thm:an}
Let $\WS$ and $\WSp$ be Coxeter systems with Coxeter matrices $M$ and $M'$, respectively. Let $a_k=a_k^{\WS}$ and $a'_k=a_k^{\WSp}$ be the growth sequences with respect to the Coxeter generating set. 
If $\WS\cleq \WSp$ then $a_k\leq a'_k$ for all $k\in \NN_0$.
\end{thmx}
\begin{proof}
Let $\phi \colon S\to S'$ be an injective map realising the relation $\WS\cleq \WSp$. Let $S''=\image\phi\subseteq S'$, let $W''=\langle S''\rangle\leq W'$, and let $(W'',S'')$ be the corresponding parabolic subsystem of $\WSp$. 
Let $\psi\colon S\to S''$ be given by $\psi(s)=\phi(s)$ for all $s\in S$. Therefore $\phi=\iota\circ\psi$, where $\iota$ is the inclusion $S''\subseteq S'$, and hence one has $\WS\cleq (W'',S'')\cleq \WSp$. 

Let $a''_k=a_k^{(W'',S'')}$. Since $(W'',S'')$ is a parabolic subgroup of $\WSp$,  then $\ell_{(W'',S'')}=\ell_{(W',S')}|_{W''}$, by \eqref{eq:l}. 
Hence $A_k^{(W'',S'')}\subseteq A_k^{(W',S')}$, and then 
\begin{equation}\label{eq:aiiai}
a''_k\leq a'_k\quad\text{ for all }k\in \NN_0.
\end{equation}

We will now prove that $a_k\leq a''_k$ for all $k$. 
Let $M''$ be the Coxeter matrix of $(W'',S'')$, and let $N=(m''_{\psi(s),\psi(r)})_{s,r\in S}$.
Since $\psi$ is a bijection, $(W(N),S)$ is Coxeter-isomorphic to $(W'',S'')$, and in particular it has growth sequence $a_k^{(W(N),S)}=a_k''$. 
Let $B_k$ and $B^N_k$ be the balls of radius $k$ in $\Cay\WS$ and $\Cay(W(N),S)$, respectively.

By hypothesis $m_{s,r}\leq n_{s,r}$ for all $s,r\in S$, and suppose that $N\neq M$. 
Without loss of generality, assume there exists a unique $2$-subset $\{\so,\ro\}\subseteq S$ such that $m_{\so,\ro}<n_{\so,\ro}$. 
Let $m=m_{\so,\ro}$ and $n=n_{\so,\ro}$.

\noindent\textbf{Claim.} For all $k$, the map $\eta_k= \left.\pi_N \sigma_M\right|_{B_k} \colon B_k\to B_k^N$, where $\sigma_M$ and $\pi_N$ are defined as in \S\ref{ss:word-pbl}, is well defined and injective. 

\noindent\textit{Proof of the claim.} 
First, notice that $\pi_N\sigma_M(B_k)\subseteq B_k^N$, since $\deg(\sigma_M(w))=\ell_M(w)$ by \eqref{eq:length-section} and $\ell_N(\pi_N(\sigma_M(w)))\leq \ell_M(w)$ by \eqref{eq:length-deg}.
Hence $\eta_k$ is well defined.

Suppose $v,v'\in B_k$ satisfy $\eta_k(v)=\eta_k(v')$ and let $w=\eta_k(v)\in B_k^N$. Thus
\[\pi_N\sigma_M(v)=\pi_N\sigma_M(v')=w=\pi_N\sigma_N(w).\]
Then, by Cor.~\ref{cor:misc-R}, \ref{li:7}, there exist sequences of $N$-moves
\begin{equation}\label{eq:Nmoves}
\sigma_M(v)\longmapsto \sigma_N(w) \longmapsfrom \sigma_M(v').
\end{equation}
Consider first the sequence $\sigma_M(v)\longmapsto \sigma_N(w)$ on the left, and suppose it can be written as the concatenation of elementary $N$-moves
\begin{equation}\label{eq:Nmove-seq}
\sigma_M(v)=\bou_0\stackrel{\nu_0}{\longmapsto}\bou_1 
\stackrel{\nu_1}{\longmapsto}\bou_2
{\longmapsto}\dots {\longmapsto}\bou_r
\stackrel{\nu_r}{\longmapsto}\bou_{r+1}=\sigma_N(w).
\end{equation}
Assume by contradiction, that there exists some $t$ for which $\nu_t$ is the $N^{(2)}$-move 
\begin{equation}\label{eq:badmove}
\nu_t\colon \bou_t =\bou'[\so,\ro,n]\bou'' \longmapsto \bou'[\ro,\so,n]\bou''=\bou_{t+1},
\end{equation}
hence  $n<\infty$, by \eqref{eq:Mmoves}. Let $t_0$ be the minimum of such $t$'s.
Thus, the sequence of moves $\nu_{t_0-1}\circ \nu_{t_0-2}\circ\dots\circ \nu_1\circ\nu_0$ is a sequence of $M$-moves transforming $\sigma_M(v)\in \caR_M(v)$ into $\bou_{t_0}$. Hence, by Cor.~\ref{cor:misc-R},~\ref{li:sigma}--\ref{li:closure}, $\bou_{t_0}\in \caR_M(v)$. 
Since $n>m$ the word $\bou_{t_0}$ has a subword of the form $[\so,\ro,m+1]$.
Therefore one may apply the $M$-moves 
\[\begin{split}
\bou_{t_0}&=\bou'[\so,\ro,m+1]\bou''=\bou'\so[\ro,\so,m]\bou'' \stackrel{M^{(2)}}{\longmapsto} \bou'\so[\so,\ro,m]\bou'' =\\ 
&=\bou'\so\so[\ro,\so,m-1]\bou'' \stackrel{M^{(1)}}{\longmapsto} \bou'[\ro,\so,m-1]\bou''=\bou''',
\end{split}\]
and hence $\deg(\bou_{t_0})> \deg\bou'''$, against the hypothesis that $\bou_{t_0}$ is ($M$-)reduced. 
This gives the desired contradiction and, thus, no $N^{(2)}$-move of the form \eqref{eq:badmove} can occurr. 
Since all the remaining $N$-moves are also $M$-moves, the sequence \eqref{eq:Nmove-seq} only consists of $M$-moves.
An analogous argument applies to the sequence $\sigma_M(v')\longmapsto \sigma_N(w)$. 
Hence, the sequences in \eqref{eq:Nmoves} entirely consist of $M$-moves, and by \eqref{eq:piM}
\[v=\pi_M\sigma_M(v)=\pi_M\sigma_N(w)=\pi_M \sigma_M(v')=v',\]
which proves the injectivity of $\eta_k$.

\smallskip
Let now $v\in A_k^{\WS}\subseteq B_k$. Then $\deg(\sigma_M(v))=k$ and the previous argument shows that $\sigma_M(v)$ is also $N$-reduced, therefore $\ell_N(\eta_k(v))=\ell_N(\pi_N\sigma_M(v))=k$.

It follows that the maps $\theta_k=\eta_k|_{A_k}\colon A_k^{\WS}\to A^{(W(N),S)}_k$ are well defined injections, and hence
\begin{equation}\label{eq:aaii}
a_k\leq a_k^{(W(N),S)}=a''_k\quad\text{ for all }k\in \NN_0.
\end{equation}
This, together with \eqref{eq:aiiai}, completes the proof.
\end{proof}

Theorem~\ref{thm:an} has the following immediate consequence.
\begin{cor}\label{cor:omega}
If $\WS$ and $\WSp$ are Coxeter systems such that $\WS\cleq\WSp$, then  \[\omega\WS\leq \omega\WSp.\]
\end{cor}

\subsection{Minimal non-spherical, non-affine Coxeter systems}\label{ss:min}

Let $\CX$ be the set of (Coxeter-isomorphism classes of) non-affine, non-spherical, irreducible Coxeter systems, and let  $\CM =\min_{\cleq} \CX$ be the set of $\cleq$-minimal elements of $\CX$.

It is well known that hyperbolic Coxeter systems are characterised as those systems such that every proper irreducible parabolic subsystem is either of spherical or affine type (cf.~\cite[Ch.~V, \S4, Ex.~13]{bourbaki--gal46}). 
By minimality, $\CM$ consists of hyperbolic Coxeter systems, which are classified in an infinite family of rank-three systems, and $72$ exceptions of rank $|S|\geq 4$ (cf.~\cite[\S\S6.8--6.9]{humphreys--rgcg}).
The infinite family consists of the $\langle a,b,c\rangle$-triangle groups with $\frac1a+\frac1b+\frac1c<1$, and among those only the $\langle2,3,7\rangle$, $\langle3,3,4\rangle$ and $\langle2,4,5\rangle$-triangle groups are $\cleq$-minimal. 
Among the $72$ exceptions, $35$ are in $\CM$. 
Therefore,

\begin{pro}[{\cite[Thm.~6.6, Tbl. 5]{mcmullen--cgsnhm}}]\label{pro:mcmullen}
  $|\CM|=38$.
\end{pro}

\section{The minimal growth rate of Coxeter groups}\label{s:tau}

Following the notation of \cite{gross-etal--cfcp}, let $E_{10}$ be the Coxeter system with Coxeter graph 
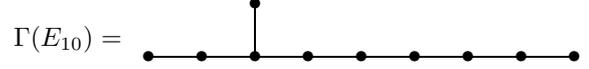
\[ \raisebox{4pt}{$\Gamma(E_{10})=$}\quad\CoxGrHNCLVI\,\,\,.
\]

\begin{thmx}\label{thm:tau}
If $\WS$ is a non-spherical, non-affine Coxeter system, then its growth rate satisfies
\[
\omega\WS\geq \tau=1.138078743\dots,
\]
where $\tau$ is the growth rate of the hyperbolic Coxeter system $E_{10}$. 
In particular, $\tau$ is the inverse of the smallest positive real root of the denominator of the Poincar\'e series $p_{E_{10}}(t)$ of the Coxeter system $E_{10}$.
Moreover, $\tau$ is an algebraic integer of degree $127$ over $\QQ$, with minimal polynomial
\[\begin{split}
m_\tau(t)=&\,t^{127} - t^{125} - t^{120} + t^{118} - t^{116} - t^{115} + t^{109} + t^{106} + t^{103} + t^{102} + 2t^{101} + t^{100} + t^{97} \\& + t^{96} + t^{91} - t^{90} - 2t^{89} - t^{88} - t^{87} - t^{86} - t^{85} - 2t^{84} - 2t^{83} - t^{82} - 2t^{81} - 3t^{80} - t^{79}\\& - t^{78} - 2t^{77} - t^{76} - t^{75} - t^{74} - t^{72} - t^{71} + t^{70} + t^{69} + 2t^{67} + 2t^{66} + t^{65} + 2t^{64} + 2t^{63}\\& + 2t^{62} + 3t^{61} + 2t^{60} + 2t^{59} + 3t^{58} + 3t^{57} + 2t^{56} + 2t^{55} + 2t^{54} + t^{53} + 2t^{52} + 2t^{51}\\& + t^{46} - t^{45} - 2t^{44} - t^{43} - t^{42} - 2t^{41} - 2t^{40} - 2t^{39} - 2t^{38} - 2t^{37} - 2t^{36} - 2t^{35} - t^{34}\\& - 2t^{33} - 3t^{32} - t^{31} - t^{29} - t^{28} - t^{27} + t^{25} + t^{22} + t^{21} + t^{20} + t^{19} + t^{18} + t^{17} + t^{16}\\& + t^{15} + t^{14} + t^{13} + t^{12} - t - 1.
\end{split}\]
The integer $\tau$ is a Perron number, i.e., an algebraic integer whose module stricly exceeds the module of its algebraic conjugates (cf.~\cite{lind--eftms,lind--etmsrcai}).
\end{thmx}
\begin{proof}
By monotonicity of the function $\omega$ with respect to $\cleq$  (cf.~Cor.~\ref{cor:omega}) and by Prop.~\ref{pro:mcmullen}, it suffices to compute $\omega\WS$ for finitely many $\WS$.

Moreover, $p_{\WS}(t)$ is power series with non-negative coefficients, and also a rational function, by Prop.~\ref{pro:ps}. 
Thus, $\omega\WS$ is the inverse of the minimal, positive real root of the denominator of $p_{\WS}(t)$.
\end{proof}
Theorem~\ref{thm:tau} can be stated in terms of a gap in the set 
\[\Omega=\{\omega\WS \mid \WS \text{ Coxeter system}\,\}\subseteq \{0,1\}\cup \RR_{\geq \tau}.\]

\begin{remark}\label{rem:notes-tau}
\begin{renum}
\item The direct verifications for the $38$ relevant Coxeter systems were performed (cf.~\cite{terragni--dhcs}) with the help of the computational algebra system Magma. The code is available at \url{https://sites.google.com/site/tomterragni/research/computations}.
\item The denominator of $p_{E_{10}}(t)$ is $(t-1)m_{\tau^{-1}}(t)$.
\item In many cases $\omega\WS$ is an algebraic integer, and also a Perron number. 
It is known that every Perron number $\lambda$ is realised as the Perron--Frobenius eigenvalue of an aperiodic, non-negative integral matrix $P_\lambda$ (cf.~\cite[Thm.~1]{lind--etmsrcai}).
Lind's proof is constructive, however the algorithm given in the proof may produce a Perron--Frobenius matrix of non-minimal size. It would be interesting to find a minimal-sized Perron--Frobenius matrix for $\tau$. 
\item The Poincar\'e series of (all but one) exceptional hyperbolic Coxeter systems are also listed in \cite{chapovalov-leites-stekolshchik--pshcgfvfd}. In the same paper, some radii of convergence are computed.
\item It is quite surprising that $\tau$ is not realised as growth rate of any of the small rank Coxeter systems, instead it is associated with the Coxeter system $E_{10}$.
However, the growth rate of one of the $\cleq$-minimal rank-three hyperbolic Coxeter groups, namely the one with Coxeter system $\langle 2,3,7\rangle$, is Lehmer's number $\lambda_{\text{Lehmer}}=1.17\dots$ (cf.~\cite{hironaka--wiln}), and an interesting coincidence occurrs.
Let
\[\lambda_\rho(W,S)= \inf\left( \{\lambda_\rho (w) \mid w\in W\}\cap \RR_{>1}\right),\]
where $\lambda_\rho(w)$ is the spectral radius of the matrix $\rho(w)$, and $\rho$ is Tits' reflection representation.

The number $\lambda_\rho(W,S)$ represents a universal bound for eigenvalues of elements in Coxeter groups. Moreover, if $\WS$ is hyperbolic, then $\log\lambda_\rho(W,S)$ is interpreted as a lower bound for the length of non-degenerate, closed hyperbolic geodesics in the orbifold $\HH^{|S|-1}/W$.

McMullen proved that 
\[\inf_{\WS}\lambda_\rho \WS=\lambda_{\text{Lehmer}},\]
the infimum being taken as $\WS$ runs through the non-affine, non-spherical Coxeter systems (cf.~\cite{mcmullen--cgsnhm}).
The infimum is actually a minimum, and it is attained \emph{exactly} for the Coxeter system $E_{10}$.

It would be interesting to understand this phenomenon.
\end{renum}
\end{remark}

\section{Rigidity and growth}\label{s:rig-growth}

It is well known that there exist non Coxeter-isomorphic Coxeter systems for which the groups are abstractly isomorphic. 
For a discussion on the isomorphism problem for Coxeter groups, see \cite{charney-davis--wcsdcg,muehlherr--ipcg,bahls--ipcg}, and references therein.

\subsection{Coxeter generating systems}\label{ss:cox-gensyst}
Let $G$ be a group generated by a finite set of involutions $R\subseteq G$. 
Then $M(R)=\left(\ord(sr)\right)_{s,r\in R}$ is a Coxeter matrix. 
Let $(W,R)$ be the Coxeter system with Coxeter matrix $M(R)$. 
The identity on $R$ induces a surjective homomorphism of groups $j_R\colon W\to G$. 
Moreover, when $j_R$ is an isomorphism $G$ is a \emph{Coxeter group with Coxeter generating system $R$}.

If $\WS$ is a Coxeter system and $\sigma$ is either an inner automorphism or the automorphism of $W$ induced by a Coxeter automorphism of $\WS$, then $\sigma(S)$ is another Coxeter generating system, and $(W,\sigma(S))$ is Coxeter-isomorphic to $\WS$. 
In general, any inner-by-Coxeter automorphism preserves the Coxeter-isomorphism type.
An automorphism which is not inner-by-Coxeter will be called \emph{exotic}.

\subsection{Isomorphisms of Coxeter groups}\label{ss:iso}

A major problem in the theory of Coxeter groups is to find all possible Coxeter generating systems of a given a Coxeter group $W$. 
If, for any two Coxeter generating sets $R,S$ of $W$, the Coxeter systems $\WS$ and $(W,R)$ are Coxeter-isomorphic, then $W$ is called \emph{rigid}. It is well known that there exist non-rigid Coxeter groups, e.g., for $n,m$ odd there are exotic isomorphisms
\begin{equation}\label{eq:exceptional-iso}
W(I_2(2m))\simeq W(I_2(m)\times A_1),\quad\text{ and }\quad W(B_n)\simeq W(D_n\times A_1).
\end{equation} 
There are standard procedures which realise exotic isomorphisms between Coxeter systems, e.g., Brady \emph{et al.} introduced the \emph{diagram twisting} (cf.~\cite[\S4]{brady-etal--rcgag} and \S\ref{s:rig-growth}), and Howlett and M\"uhlherr introduced a construction, the \emph{elementary reductions}, which deal with exotic isomorphisms $(W,S)\to (W,R)$ for which the set of reflections $S^W$ is different from $R^W$ (cf.~\cite{howlett-muehlherr--icgwdnpr}). Reductions generalise the exotic isomorphisms \eqref{eq:exceptional-iso}.

Several classes of Coxeter groups are known to be rigid or rigid up to diagram twisting. For instance, if any of the following conditions is satisfied for a Coxeter generating system $S$ of $W$, then $W$ is rigid up to diagram twisting (cf.~\cite{brady-etal--rcgag, bahls--ipcg,muehlherr--ipcg}).
\begin{renum}
\item $\WS$ is right-angled, i.e., $m_{s,r}\in \{2,\infty\}$ for all $s,r\in S$, $s\neq r$;
\item $\WS$ is infinite and $m_{s,r}<\infty$ for all $s,r\in S$;
\item $\WS$ can act faithfully, properly and cocompactly on a contractible manifold;
\item $\WS$ is skew-angled, i.e., $m_{s,r}\neq 2$ for all $s,r\in S$;
\item $\Gamma_\infty\WS$ is a tree, where $\Gamma_\infty$ is the variant of the Coxeter graph defined in \cite[Ch.~IV, \S1, Ex. 11]{bourbaki--gal46}.
\end{renum}

\subsection{Mutations of Coxeter groups}
\begin{dfn}\label{dfn:twist}
Let $M$ be a Coxeter matrix over $S$, and suppose that there exists a partition $S=X\sqcup Y \sqcup T \sqcup Z$ and a Coxeter-automorphism $\sigma$ of the subsystem $(W_X,X)$ satisfying
\begin{renum}
\item\label{li:TY} $m_{t,y}=\infty$ for all $t\in T$ and $y\in Y$,  
\item\label{li:ZY} $m_{z,y}<\infty$ for all $z\in Z$ and $y\in Y$, and
\item\label{li:ZX} for all $z\in Z$ and $x\in X$ one has $m_{z,\sigma(x)}=m_{z,x}$.
\end{renum}
Then, the $4$-tuple $(M,X,Y,\sigma)$ is called \emph{mutable}.
Associated with a mutable tuple $(M,X,Y,\sigma)$ there is a Coxeter matrix $\mu(M,X,Y,\sigma)=(n_{r,s})_{r,s\in S}$, its \emph{mutation}, given by 
\begin{equation}\label{eq:mutation}
n_{s,r}=n_{r,s}=\begin{cases}
m_{\sigma(r),s}& \text{ if } r\in X, s\in Y,\\
m_{\sigma(r),\sigma(s)}& \text{ if } r,s\in X,\\
m_{r,s}& \text{ otherwise.}
\end{cases}\end{equation}

If $(M,X,Y,\sigma)$ is mutable, then $(\mu(M,X,Y,\sigma), X,Y, \sigma^{-1})$ is mutable and it is called the \emph{inverse mutable} $4$-tuple since $\mu(\mu(M,X,Y,\sigma), X,Y, \sigma^{-1})=M$. 
The relation ``$N$ is a mutation of $M$'' is symmetric, and  therefore its transitive closure is an equivalence relation $\sim$ on Coxeter systems.
\end{dfn}
\begin{remark}\begin{renum}
\item The partition associated with a mutable tuple $(M,X,Y,\sigma)$ is determined by $X$, $Y$ together with conditions \ref{li:TY}--\ref{li:ZY}, and therefore $T,Z$ may be omitted from the notation.
\item Many Coxeter matrices $M$ only admit trivially mutable tuples, i.e., tuples with $\sigma=\iden_X$. 
Even when a non-trivial tuple exists, it may happen that the associated mutation is Coxeter-isomorphic to $M$. If this is not the case, $(M,X,Y,\sigma)$ is called \emph{effective}.
\item The operation of mutation is a generalisation of the diagram twisting (cf.~\cite{brady-etal--rcgag}). Diagram twists are  mutations satisfying the additional conditions (a) $W_X$ is finite, (b) $\sigma(x)=x^{w_0(X)}$ is the conjugation by the longest element of $W_X$, and (c) $m_{z,x}=2$ for all $z\in Z$ and $x\in X$. 
Effective diagram twists determine exotic isomorphisms of Coxeter groups.
\end{renum}
\end{remark}

\begin{thmx}\label{thm:mu}
Let $\WS$ be a Coxeter system with Coxeter matrix $M$, and let $(M,X,Y,\sigma)$ be a mutable tuple for $\WS$.
Let $N= \mu(M,X,Y,\sigma)$, and let $(W',S')$ be the Coxeter system with Coxeter matrix $N$. 

Then there is a bijection $\argu^\sharp\colon \spr=\spr\WS\to \spr(W',S')=\spr'$, such that $(W_I,I)$ is Coxeter-isomorphic to $(W'_{I^\sharp},I^\sharp)$ for all $I\in \spr$. 
Moreover, if $\WS \sim \WSp$ then
\begin{equation}\label{eq:pmutation}
p_{\WS}(t)=p_{(W',S')}(t). 
\end{equation}
\end{thmx}
\begin{proof}
Let $S=X \sqcup Y \sqcup Z \sqcup T $ decompose as in Def.~\ref{dfn:twist}, and let $I\in\spr$. Since every edge of a spherical graph must have a finite label, then either 
\begin{itemize}
\item[(a)] $I\subseteq X\sqcup T \sqcup Z$, or 
\item[(b)] $I\subseteq X\sqcup Y\sqcup Z$ and $I\cap Y\neq \emptyset$.
\end{itemize}
Suppose that (a) holds, then define $I^\sharp=\{r^\sharp=r\mid r\in I\}$. By \eqref{eq:mutation}, for $r^\sharp, s^\sharp\in I^\sharp$ on has
\[n_{r^\sharp,s^\sharp}=n_{r,s}=\begin{cases} 
 m_{\sigma(r),\sigma(s)} &\text{ if } r,s\in X,\\
 m_{r,s} &\text{ if } r\in X, s\not\in X,\\
 m_{r,s} &\text{ if } r,s\not\in X.
\end{cases}\]
Since $\sigma$ is a Coxeter-automorphism of $(W_X,X)$, then $m_{\sigma(r),\sigma(s)}=m_{r,s}$ for $s,r\in X$.

\medskip
Suppose that (b) holds, then define $I^\sharp=\{r^\sharp \mid r\in I\}$, where now
\begin{equation}\label{eq:sharp}
r^\sharp=\begin{cases} 
\sigma^{-1}(r) &\text{ if } r\in X,\\
r &\text{ if } r\not\in X.\end{cases}
\end{equation}
Then, for $r^\sharp, s^\sharp\in I^\sharp$, by \eqref{eq:mutation}, \eqref{eq:sharp} and Def.~\ref{dfn:twist},~\ref{li:ZX}, one has
\[n_{r^\sharp,s^\sharp}=\begin{cases} 
 m_{\sigma(r^\sharp),\sigma(s^\sharp)}=m_{r,s} &\text{ if } r^\sharp,s^\sharp\in X ,\\
 m_{\sigma(r^\sharp),s^\sharp}=m_{r,s} &\text{ if } r^\sharp\in X, s^\sharp\in Y ,\\
 m_{r^\sharp,s^\sharp}=m_{\sigma^{-1}(r),s}=m_{r,s} &\text{ if } r^\sharp\in X, s^\sharp\in Z,\\
 m_{r^\sharp,s^\sharp}=m_{r,s} &\text{ if } r,s\not\in X.
\end{cases}\]

Hence, $N_{I^\sharp}$ and $M_I$ determine Coxeter-isomorphic systems. 
It follows that $I^\sharp\in \spr'$ and that (a)  holds for $I^\sharp$ if, and only if, (a) holds for $I$.
Thus, the map $I\mapsto I^\sharp$ is a map which preserves the Coxeter-isomorphism type, and it is invertible (its inverse being the $\sharp$-map associated to the inverse mutable tuple). 
The identity \eqref{eq:pmutation} then follows from Steinberg's formula \eqref{eq:sph-steinberg}.
\end{proof}

\begin{cor}\label{cor:mgr-cox}
Suppose that $W$ is rigid up to diagram twisting, and let $S,R$ be Coxeter generating systems for $W$ (cf.~\S\ref{ss:cox-gensyst}). Then 
\[p_{\WS}(t)=p_{(W,R)}(t)\quad \text{ and }\quad\omega\WS=\omega(W,R).\]
Let $p_{W,\textup{Cox}}(t)$ and $\omega_{\textup{Cox}}(W)$ be these common values.
\end{cor}

Theorem \ref{thm:mu} implies that effective mutations which are not diagram twists can be regarded as procedures to produce non-isomorphic (and \emph{a fortiori}, non Coxeter-isomorphic) Coxeter systems with the same Poicar\'e series.

\begin{exa}Consider the rank-seven Coxeter system $\WS$ with Coxeter matrix
\[M=\begin{pmatrix}
1 & 3 & 3 & 2 & 3 & 4 & 2\\
3 & 1 & 3 & 2 & 3 & 4 & 2\\
3 & 3 & 1 & 2 & 2 & 4 & 3\\
2 & 2 & 2 & 1 & 3 & 3 & 2\\
3 & 3 & 2 & 3 & 1 & 2 & \infty\\
4 & 4 & 4 & 3 & 2 & 1 & 3\\
2 & 2 & 3 & 2 & \infty & 3 & 1
\end{pmatrix}.
\]

Let $X =\{s_1,s_2,s_3,s_4\}$, $Y=\{s_5\}$, $Z=\{s_6\}$, $T=\{s_7\}$, and let $\sigma=(1,2,3)$. 
Then $(M,X,Y,\sigma)$ is mutable, with mutation displayed in~Fig.~\ref{fig:mut}.
Moreover, $N=\mu(M,X,Y,\sigma)$ is a proper mutation, i.e., $N$ is not obtained from $M$ by diagram twisting. 

\begin{figure}[htbp]\begin{center}
\begin{tikzpicture}[auto,inner sep=0.5mm,
    vertex/.style={circle,draw=black,minimum size=4mm}]
  \node[vertex] (1) at (3,4) {$\scriptstyle s_1$};
  \node[vertex] (2) at (4,3) {$\scriptstyle s_2$};
  \node[vertex] (3) at (5,4) {$\scriptstyle s_3$};
  \node[vertex] (4) at (1.5,4) {$\scriptstyle s_4$};
  \node[vertex] (5) at (1,2) {$\scriptstyle s_5$};
  \node[vertex] (6) at (4,1) {$\scriptstyle s_6$};
  \node[vertex] (7) at (6,2) {$\scriptstyle s_7$};
  \draw (1) to (2);
  \draw (1) to (3);
  \draw (1) to (5);
  \draw (1) to node [pos=0.2,swap] {$\scriptstyle 4$} (6);
  \draw (2) to (3);
  \draw (2) to (5);
  \draw (2) to node [pos=0.3] {$\scriptstyle 4$} (6);
  \draw (3) to node[above,right] {$\scriptstyle 4$} (6);
  \draw (3) to (7);
  \draw (4) to (5);
  \draw (4) to (6);
  \draw (5) to node [pos=0.25,swap] {$\scriptstyle \infty$} (7);
  \draw (6) to (7);
  \node at (4,3.6) {$\scriptstyle\sigma$};
  \draw[ <-] (3.8,3.8) arc [start angle=135, end angle=405, x radius=3mm, y radius=2mm];
  \draw [->](6.5,2.5) -- node {$\mu$} (8,2.5);
\begin{scope}[xshift=7.5cm]
  \node[vertex] (1) at (3,4) {$\scriptstyle s_2$};
  \node[vertex] (2) at (4,3) {$\scriptstyle s_3$};
  \node[vertex] (3) at (5,4) {$\scriptstyle s_1$};
  \node[vertex] (4) at (1.5,4) {$\scriptstyle s_4$};
  \node[vertex] (5) at (1,2) {$\scriptstyle s_5$};
  \node[vertex] (6) at (4,1) {$\scriptstyle s_6$};
  \node[vertex] (7) at (6,2) {$\scriptstyle s_7$};
  \draw (1) to (2);
  \draw (1) to (3);
  \draw (1) to (5);
  \draw (1) to node [pos=0.2,swap] {$\scriptstyle 4$} (6);
  \draw (2) to (3);
  \draw (2) to (5);
  \draw (2) to node [pos=0.3] {$\scriptstyle 4$} (6);
  \draw (3) to node[above,right] {$\scriptstyle 4$} (6);
  \draw[bend left] (2) to (7);
  \draw (4) to (5);
  \draw (4) to (6);
  \draw (5) to node [pos=0.25,swap] {$\scriptstyle \infty$} (7);
  \draw (6) to (7);
\end{scope}
\end{tikzpicture}\end{center}
\caption{A proper mutation.}\label{fig:mut}
\end{figure}
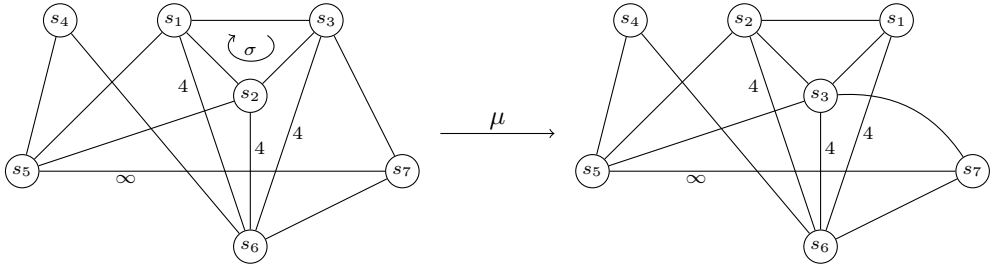
\end{exa}

\subsection{A conjecture}\label{ss:conj}
Consider the group $\PGL(2,\ZZ)\simeq (C_2\times C_2)\ast_{C_2}S_3$. It is well known that $\PGL(2,\ZZ)\simeq W$, where $(W,S)$ is the Coxeter system $\langle 2,3,\infty\rangle$ with Coxeter graph $\,\CoxGrHCI{}{\infty}\,$. Hence the minimal growth rate satisfies $\omega(\PGL(2,\ZZ))\leq \omega(W,S)=\alpha$, where $\alpha$ is the \emph{plastic number}, with minimal polynomial $m_\alpha(t)=t^3-t-1$. The converse inequality is proven by Bucher and Talambutsa (cf.~\cite[\S6]{bucher-talambutsa--egrfap}).

Therefore, the following problem seems to be of some interest.
\begin{conjx}\label{conj}
Let $W$ be a Coxeter group rigid up to diagram twisting, and let $\omega_{\textup{Cox}}(W)$ be defined as in Cor.~\ref{cor:mgr-cox}. 
Then $\omega(W)=\omega_{\textup{Cox}}(W)$.
\end{conjx}

\begin{remark}
  \begin{renum}
  \item If $W$ is a product of spherical and affine irreducible Coxeter systems, its Poincar\'e series depends on the chosen generating set. However, the minimal growth rate and the growth rate coincide $\omega(W)=\omega(W,S)$ and their common value is either $0$ or $1$, depending on the finiteness of the group only. 
  \item The rigidity hypothesis in Conj.~\ref{conj} cannot be relaxed since, in general, elementary reductions do not preserve the growth rate, as the following example shows. Let 
\[M=\begin{pmatrix}
1 & 3 & 2 & 3 & \infty\\
3 & 1 & 2 & 2 & 2\\
2 & 2 & 1 & 3 & 2\\
3 & 2 & 3 & 1 & 4\\
\infty & 2 & 2 & 4 & 1
\end{pmatrix},\quad\quad \Gamma(M)=\,\,
\raisebox{-24pt}{\begin{tikzpicture}[auto,inner sep=0.5mm,
    vertex/.style={circle,draw=black,minimum size=4mm}]
  \node[vertex] (1) at (0,0) {$\scriptstyle s_1$};
  \node[vertex] (2) at (0,1.5) {$\scriptstyle s_2$};
  \node[vertex] (3) at (3,0) {$\scriptstyle s_3$};
  \node[vertex] (4) at (1.5,0) {$\scriptstyle s_4$};
  \node[vertex] (5) at (1.5,1.5) {$\scriptstyle s_5$};
  \draw (1) to (2);
  \draw (1) to (4);
  \draw (1) to node {$\scriptstyle \infty$} (5);
  \draw (3) to (4);
  \draw (4) to node[above,right] {$\scriptstyle 4$} (5);
\end{tikzpicture}}.
\]
Then $s_5$ is a pseudo-transposition, corresponding to the parabolic subsystem of type $B_3$ generated by $J=\{s_3,s_4,s_5\}$. 
Let $r_i=s_i$ for $i \in \{1,\dots,4\}$, let $r_5=s_5s_4s_5$ and let $r_6=w_0(J)=s_3s_4s_3s_5s_4s_3s_5s_4s_5$ be the longest element of the parabolic subsystem $(W_J,J)$. Then, $R=\{r_i\mid i\in \{1,\dots,6\}\,\}$ is a Coxeter generating system for $W(M)$ (cf.~\cite{howlett-muehlherr--icgwdnpr}). Its Coxeter matrix $M'=M(R)$ is 
\[M'=\begin{pmatrix}
1 & 3 & 2 & 3 & \infty & \infty\\
3 & 1 & 2 & 2 & 2 & 2\\
2 & 2 & 1 & 3 & 3 & 2\\
3 & 2 & 3 & 1 & 2 & 2\\
\infty & 2 & 3 & 2 & 1 & 2\\
\infty & 2 & 2 & 2 & 2 & 1
\end{pmatrix},\quad\quad \Gamma(M')=\,\,
\raisebox{-24pt}{\begin{tikzpicture}[auto,inner sep=0.5mm,
    vertex/.style={circle,draw=black,minimum size=4mm}]
  \node[vertex] (1) at (1.5,0) {$\scriptstyle r_1$};
  \node[vertex] (2) at (1.5,1.5) {$\scriptstyle r_2$};
  \node[vertex] (3) at (4.5,0) {$\scriptstyle r_3$};
  \node[vertex] (4) at (3,0) {$\scriptstyle r_4$};
  \node[vertex] (5) at (3,1.5) {$\scriptstyle r_5$};
  \node[vertex] (6) at (0,0) {$\scriptstyle r_6$};
  \draw (1) to (2);
  \draw (1) to (4);
  \draw (1) to node  {$\scriptstyle \infty$} (5);
  \draw (1) to node [swap] {$\scriptstyle \infty$} (6);
  \draw (3) to (4);
  \draw (4) to node[above,right] {$\scriptstyle 4$} (5);
\end{tikzpicture}}.
\]
By direct computation one sees that $\omega(W,S)= 2.24167\dots$, while $\omega(W,R)=2.61578\dots$.
  \end{renum}
\end{remark}

\bibliographystyle{amsalpha}
\bibliography{all-references,terragni-references}

\end{document}

%% file: graphs-defs.tex
\newcommand{\CoxBullet}{\bullet}

\newcommand{\CircuitN}{%
\begin{xy} 0;<0.7cm,0cm>: 
(0,0)*=0{\CoxBullet}="1" , 
(1,0)*=0{\CoxBullet}="2" , 
(2,0)*=0{\CoxBullet}="3" ,
(3,0)*=0{\CoxBullet}="4" , 
(1.5,1)*=0{\CoxBullet}="5" , 
\ar@{-} "1";"2",
\ar@{..} "2";"3"
\ar@{-} "3";"4"
\ar@{-} "4";"5"
\ar@{-}"5";"1"
\end{xy}
}
\newcommand{\Linn}[2]{%
\begin{xy} 0;<0.7cm,0cm>: 
(0,0)*=0{\CoxBullet}="1" , 
(1,0)*=0{\CoxBullet}="2" , 
(2,0)*=0{\CoxBullet}="3" ,
(3,0)*=0{\CoxBullet}="4" , 
(4,0)*=0{\CoxBullet}="5" , 
(5,0)*=0{\CoxBullet}="6" ,
\ar@{-}^{#1} "1";"2",
\ar@{-} "2";"3"
\ar@{..} "3";"4"
\ar@{-} "4";"5"
\ar@{-}^{#2} "5";"6"
\end{xy}
}

\newcommand{\Lintwo}[1]{%
\begin{xy} 0;<0.7cm,0cm>: 
(0,0)*=0{\CoxBullet}="1" , 
(1,0)*=0{\CoxBullet}="2" , 
\ar@{-}^{#1} "1";"2",
\end{xy}
}

\newcommand{\LinT}[2]{%
\begin{xy} 0;<0.7cm,0cm>: 
(0,0)*=0{\CoxBullet}="1" , 
(1,0)*=0{\CoxBullet}="2" , 
(2,0)*=0{\CoxBullet}="3" ,
\ar@{-}^{#1} "1";"2",
\ar@{-}^{#2} "2";"3"
\end{xy}
}

\newcommand{\Tri}[3]{%
\begin{xy} 0;<0.7cm,0cm>: 
(0,0)*=0{\CoxBullet}="1" , 
(1,0)*=0{\CoxBullet}="2" , 
(0.5,0.866)*=0{\CoxBullet}="3" ,
\ar@{-}^{#1} "1";"2",
\ar@{-}_{#2} "2";"3",
\ar@{-}_{#3} "3";"1",
\end{xy}
}

\newcommand{\LinQ}[3] 
{\begin{xy} 0;<0.7cm,0cm>:
(0,0)*=0{\CoxBullet}="1" , 
(1,0)*=0{\CoxBullet}="2" , 
(2,0)*=0{\CoxBullet}="3" ,
(3,0)*=0{\CoxBullet}="4",
\ar@{-}^{#1} "1";"2",
\ar@{-}^{#2} "2";"3",
\ar@{-}^{#3} "3";"4"
\end{xy}}

\newcommand{\LinP}[4] 
{\begin{xy} 0;<0.7cm,0cm>:
(0,0)*=0{\CoxBullet}="1" , 
(1,0)*=0{\CoxBullet}="2" , 
(2,0)*=0{\CoxBullet}="3" ,
(3,0)*=0{\CoxBullet}="4",
(4,0)*=0{\CoxBullet}="5",
\ar@{-}^{#1} "1";"2",
\ar@{-}^{#2} "2";"3",
\ar@{-}^{#3} "3";"4",
\ar@{-}^{#4} "4";"5"
\end{xy}}

\newcommand{\LinS}[5] 
{\begin{xy} 0;<0.7cm,0cm>:
(0,0)*=0{\CoxBullet}="1" , 
(1,0)*=0{\CoxBullet}="2" , 
(2,0)*=0{\CoxBullet}="3" ,
(3,0)*=0{\CoxBullet}="4",
(4,0)*=0{\CoxBullet}="5",
(5,0)*=0{\CoxBullet}="6",
\ar@{-}^{#1} "1";"2",
\ar@{-}^{#2} "2";"3",
\ar@{-}^{#3} "3";"4",
\ar@{-}^{#4} "4";"5",
\ar@{-}^{#5} "5";"6"
\end{xy}}

\newcommand{\QFork}[3] 
{\begin{xy} 0;<0.7cm,0cm>:
(0,0)*=0{\CoxBullet}="1" , 
(1,0)*=0{\CoxBullet}="2" , 
(1.866,0.5)*=0{\CoxBullet}="3" ,
(1.866,-0.5)*=0{\CoxBullet}="4",
\ar@{-}^{#1} "1";"2",
\ar@{-}^{#2} "2";"3",
\ar@{-}_{#3} "2";"4"
\end{xy}}

\newcommand{\PFork}[3] 
{\begin{xy} 0;<0.7cm,0cm>:
(0,0)*=0{\CoxBullet}="1" , 
(1,0)*=0{\CoxBullet}="2" , 
(2,0)*=0{\CoxBullet}="3" , 
(2.866,0.5)*=0{\CoxBullet}="4" ,
(2.866,-0.5)*=0{\CoxBullet}="5",
\ar@{-}^{#1} "1";"2",
\ar@{-}^{#2} "2";"3",
\ar@{-}^{#3} "3";"4",
\ar@{-} "3";"5"
\end{xy}}

\newcommand{\SFork}[1] 
{\begin{xy} 0;<0.7cm,0cm>:
(0,0)*=0{\CoxBullet}="1" , 
(1,0)*=0{\CoxBullet}="2" , 
(2,0)*=0{\CoxBullet}="3" , 
(3,0)*=0{\CoxBullet}="4" , 
(3.866,0.5)*=0{\CoxBullet}="5" ,
(3.866,-0.5)*=0{\CoxBullet}="6",
\ar@{-} "1";"2",
\ar@{-}^{#1} "2";"3",
\ar@{-} "3";"4",
\ar@{-} "4";"5",
\ar@{-} "4";"6"
\end{xy}}

\newcommand{\nFork}[1]{%
\begin{xy} 0;<0.7cm,0cm>:
(0,0)*=0{\CoxBullet}="1" , 
(1,0)*=0{\CoxBullet}="2" , 
(2,0)*=0{\CoxBullet}="3" , 
(3,0)*=0{\CoxBullet}="4" , 
(4,0)*=0{\CoxBullet}="5" , 
(4.866,0.5)*=0{\CoxBullet}="6" ,
(4.866,-0.5)*=0{\CoxBullet}="7",
\ar@{-}^{#1} "1";"2",
\ar@{-} "2";"3",
\ar@{.} "3";"4",
\ar@{-} "4";"5",
\ar@{-} "5";"6",
\ar@{-} "5";"7"
\end{xy}}

\newcommand{\nDoubleFork}{%
\begin{xy} 0;<0.7cm,0cm>:
(0,-0.5)*=0{\CoxBullet}="1" , 
(0,0.5)*=0{\CoxBullet}="2" , 
(0.866,0)*=0{\CoxBullet}="3" , 
(1.866,0)*=0{\CoxBullet}="4" , 
(2.866,0)*=0{\CoxBullet}="5" , 
(3.866,0)*=0{\CoxBullet}="6" , 
(4.732,0.5)*=0{\CoxBullet}="7" ,
(4.732,-0.5)*=0{\CoxBullet}="8",
\ar@{-} "1";"3",
\ar@{-} "2";"3",
\ar@{-} "3";"4",
\ar@{.} "4";"5",
\ar@{-} "5";"6",
\ar@{-} "6";"7"
\ar@{-} "6";"8"
\end{xy}}

\newcommand{\Quad}[4]
{\begin{xy} 0;<0.7cm,0cm>:
(0,0)*=0{\CoxBullet}="1" , 
(1,0)*=0{\CoxBullet}="2" , 
(1,1)*=0{\CoxBullet}="3" ,
(0,1)*=0{\CoxBullet}="4",
\ar@{-}^{#1} "1";"2",
\ar@{-}_{#2} "2";"3",
\ar@{-}_{#3} "3";"4",
\ar@{-}_{#4} "4";"1"
\end{xy}}

\newcommand{\Pent}[5]
{\begin{xy} 0;<0.7cm,0cm>:
(0,0)*=0{\CoxBullet}="1" , 
(1,0)*=0{\CoxBullet}="2" , 
(1.5,0.866)*=0{\CoxBullet}="3" ,
(0.5,0.866)*=0{\CoxBullet}="4",
(-0.5,0.866)*=0{\CoxBullet}="5",
\ar@{-}^{#1} "1";"2",
\ar@{-}_{#2} "2";"3",
\ar@{-}_{#3} "3";"4",
\ar@{-}_{#4} "4";"5",
\ar@{-}_{#5} "5";"1",
\end{xy}}

\newcommand{\QELoop}[1] 
{\begin{xy} 0;<0.7cm,0cm>:
(0,0)*=0{\CoxBullet}="1" , 
(1,0)*=0{\CoxBullet}="2" , 
(1.866,0.5)*=0{\CoxBullet}="3" ,
(1.866,-0.5)*=0{\CoxBullet}="4",
\ar@{-}^{#1} "1";"2",
\ar@{-} "2";"3",
\ar@{-} "2";"4",
\ar@{-} "3";"4"
\end{xy}}

\newcommand{\PELoop}[1] 
{\begin{xy} 0;<0.7cm,0cm>:
(0,0)*=0{\CoxBullet}="1" , 
(1,0)*=0{\CoxBullet}="2" , 
(1.866,0.5)*=0{\CoxBullet}="3" ,
(1.866,-0.5)*=0{\CoxBullet}="4",
(2.732,0)*=0{\CoxBullet}="5",
\ar@{-}^{#1} "1";"2",
\ar@{-} "2";"3",
\ar@{-} "2";"4",
\ar@{-} "4";"5",
\ar@{-} "3";"5"
\end{xy}}

\newcommand{\SELoop}
{\begin{xy} 0;<0.7cm,0cm>:
(0,0)*=0{\CoxBullet}="1" , 
(1,0)*=0{\CoxBullet}="2" , 
(1.866,0.5)*=0{\CoxBullet}="3" ,
(1.866,-0.5)*=0{\CoxBullet}="4",
(2.866,0.5)*=0{\CoxBullet}="5",
(2.866,-0.5)*=0{\CoxBullet}="6",
\ar@{-} "1";"2",
\ar@{-} "2";"3",
\ar@{-} "3";"5",
\ar@{-} "5";"6",
\ar@{-} "6";"4",
\ar@{-} "4";"2"
\end{xy}}

\newcommand{\HELoop}
{\begin{xy} 0;<0.7cm,0cm>:
(0,0)*=0{\CoxBullet}="1" , 
(1,0)*=0{\CoxBullet}="2" , 
(1.866,0.5)*=0{\CoxBullet}="3" ,
(1.866,-0.5)*=0{\CoxBullet}="4",
(2.866,0.5)*=0{\CoxBullet}="5",
(2.866,-0.5)*=0{\CoxBullet}="6",
(3.732,0)*=0{\CoxBullet}="7",
\ar@{-} "1";"2",
\ar@{-} "2";"3",
\ar@{-} "3";"5",
\ar@{-} "5";"7",
\ar@{-} "7";"6",
\ar@{-} "6";"4",
\ar@{-} "4";"2",
\end{xy}}

\newcommand{\OELoop}
{\begin{xy} 0;<0.7cm,0cm>:
(0,0)*=0{\CoxBullet}="1" , 
(1,0)*=0{\CoxBullet}="2" , 
(1.866,0.5)*=0{\CoxBullet}="3" ,
(1.866,-0.5)*=0{\CoxBullet}="4",
(2.866,0.5)*=0{\CoxBullet}="5",
(2.866,-0.5)*=0{\CoxBullet}="6",
(3.866,0.5)*=0{\CoxBullet}="7",
(3.866,-0.5)*=0{\CoxBullet}="8",
\ar@{-} "1";"2",
\ar@{-} "2";"3",
\ar@{-} "3";"5",
\ar@{-} "5";"7",
\ar@{-} "7";"8"
\ar@{-} "8";"6",
\ar@{-} "6";"4",
\ar@{-} "4";"2"
\end{xy}}

\newcommand{\NELoop}
{\begin{xy} 0;<0.7cm,0cm>:
(0,0)*=0{\CoxBullet}="1" , 
(1,0)*=0{\CoxBullet}="2" , 
(1.866,0.5)*=0{\CoxBullet}="3" ,
(1.866,-0.5)*=0{\CoxBullet}="4",
(2.866,0.5)*=0{\CoxBullet}="5",
(2.866,-0.5)*=0{\CoxBullet}="6",
(3.866,0.5)*=0{\CoxBullet}="7",
(3.866,-0.5)*=0{\CoxBullet}="8",
(4.732,0)*=0{\CoxBullet}="9",
\ar@{-} "1";"2",
\ar@{-} "2";"3",
\ar@{-} "3";"5",
\ar@{-} "5";"7",
\ar@{-} "7";"9",
\ar@{-} "9";"8",
\ar@{-} "8";"6",
\ar@{-} "6";"4",
\ar@{-} "4";"2",
\end{xy}}

\newcommand{\Tetra}
{\begin{xy} 0;<0.7cm,0cm>:
(0,0)*=0{\CoxBullet}="1" , 
(0,1)*=0{\CoxBullet}="2" , 
(-0.866,-0.5)*=0{\CoxBullet}="3" ,
(0.866,-0.5)*=0{\CoxBullet}="4",
\ar@{-} "1";"2",
\ar@{-} "2";"3",
\ar@{-} "2";"4",
\ar@{-} "3";"4",
\ar@{-} "1";"4",
\ar@{-} "3";"1"
\end{xy}}

\newcommand{\BigQ}
{\begin{xy} 0;<0.7cm,0cm>:
(0,0)*=0{\CoxBullet}="1" , 
(0.707,0.707)*=0{\CoxBullet}="2" , 
(-0.707,0.707)*=0{\CoxBullet}="3" ,
(-0.707,-0.707)*=0{\CoxBullet}="4",
(0.707,-0.707)*=0{\CoxBullet}="5",
\ar@{-} "1";"2",
\ar@{-} "2";"3",
\ar@{-} "3";"4",
\ar@{-} "4";"5",
\ar@{-} "5";"2",
\ar@{-} "4";"1"
\end{xy}}

\newcommand{\EsixL}[2]{%
\begin{xy} 0;<0.7cm,0cm>:
(0,0)*=0{\CoxBullet}="1" , 
(1,0)*=0{\CoxBullet}="2" , 
(2,0)*=0{\CoxBullet}="3" ,
(3,0)*=0{\CoxBullet}="4",
(4,0)*=0{\CoxBullet}="5",
(2,1)*=0{\CoxBullet}="6",
\ar@{-}^{#1} "1";"2",
\ar@{-} "2";"3",
\ar@{-} "3";"4",
\ar@{-}^{#2} "4";"5",
\ar@{-} "3";"6"
\end{xy}}

\newcommand{\EsevenL}[1]{%
\begin{xy} 0;<0.7cm,0cm>:
(0,0)*=0{\CoxBullet}="1" , 
(1,0)*=0{\CoxBullet}="2" , 
(2,0)*=0{\CoxBullet}="3" ,
(3,0)*=0{\CoxBullet}="4",
(4,0)*=0{\CoxBullet}="5",
(2,1)*=0{\CoxBullet}="6",
(5,0)*=0{\CoxBullet}="7",
\ar@{-} "1";"2",
\ar@{-} "2";"3",
\ar@{-} "3";"4",
\ar@{-} "4";"5",
\ar@{-} "3";"6"
\ar@{-}^{#1} "5";"7",
\end{xy}}

\newcommand{\EeightL}[1]{%
\begin{xy} 0;<0.7cm,0cm>:
(0,0)*=0{\CoxBullet}="1" , 
(1,0)*=0{\CoxBullet}="2" , 
(2,0)*=0{\CoxBullet}="3" ,
(3,0)*=0{\CoxBullet}="4",
(4,0)*=0{\CoxBullet}="5",
(2,1)*=0{\CoxBullet}="6",
(5,0)*=0{\CoxBullet}="7",
(6,0)*=0{\CoxBullet}="8",
\ar@{-} "1";"2",
\ar@{-} "2";"3",
\ar@{-} "3";"4",
\ar@{-} "4";"5",
\ar@{-} "3";"6",
\ar@{-} "5";"7",
\ar@{-}^{#1} "7";"8",
\end{xy}}

\newcommand{\TEeightL}[1]{%
\begin{xy} 0;<0.7cm,0cm>:
(0,0)*=0{\CoxBullet}="1" , 
(1,0)*=0{\CoxBullet}="2" , 
(2,0)*=0{\CoxBullet}="3" ,
(3,0)*=0{\CoxBullet}="4",
(4,0)*=0{\CoxBullet}="5",
(5,0)*=0{\CoxBullet}="6",
(6,0)*=0{\CoxBullet}="7",
(7,0)*=0{\CoxBullet}="8",
(2,1)*=0{\CoxBullet}="9",
\ar@{-} "1";"2",
\ar@{-} "2";"3",
\ar@{-} "3";"4",
\ar@{-} "4";"5",
\ar@{-} "5";"6",
\ar@{-} "6";"7",
\ar@{-}^{#1} "7";"8"
\ar@{-} "3";"9",
\end{xy}}

\newcommand{\TTEsix}{%
\begin{xy} 0;<0.7cm,0cm>:
(0,0)*=0{\CoxBullet}="1" , 
(1,0)*=0{\CoxBullet}="2" , 
(2,0)*=0{\CoxBullet}="3" ,
(3,0)*=0{\CoxBullet}="4",
(3.866,0.5)*=0{\CoxBullet}="5",
(4.866,0.5)*=0{\CoxBullet}="6",
(3.866,-0.5)*=0{\CoxBullet}="7",
(4.866,-0.5)*=0{\CoxBullet}="8",
\ar@{-} "1";"2",
\ar@{-} "2";"3",
\ar@{-} "3";"4",
\ar@{-} "4";"5",
\ar@{-} "5";"6",
\ar@{-} "4";"7",
\ar@{-} "7";"8",
\end{xy}}

\newcommand{\TEsix}{%
\begin{xy} 0;<0.7cm,0cm>:
(0,0)*=0{\CoxBullet}="1" , 
(1,0)*=0{\CoxBullet}="2" , 
(2,0)*=0{\CoxBullet}="3" ,
(2.866,0.5)*=0{\CoxBullet}="4",
(3.866,0.5)*=0{\CoxBullet}="5",
(2.866,-0.5)*=0{\CoxBullet}="6",
(3.866,-0.5)*=0{\CoxBullet}="7",
\ar@{-} "1";"2",
\ar@{-} "2";"3",
\ar@{-} "3";"4",
\ar@{-} "4";"5",
\ar@{-} "3";"6",
\ar@{-} "6";"7",
\end{xy}}

\newcommand{\TTEseven}{%
\begin{xy} 0;<0.7cm,0cm>:
(0,0)*=0{\CoxBullet}="1" , 
(1,0)*=0{\CoxBullet}="2" , 
(2,0)*=0{\CoxBullet}="3" ,
(3,0)*=0{\CoxBullet}="4",
(4,0)*=0{\CoxBullet}="5",
(5,0)*=0{\CoxBullet}="6",
(6,0)*=0{\CoxBullet}="7",
(7,0)*=0{\CoxBullet}="8",
(3,1)*=0{\CoxBullet}="9",
\ar@{-} "1";"2",
\ar@{-} "2";"3",
\ar@{-} "3";"4",
\ar@{-} "4";"5",
\ar@{-} "5";"6",
\ar@{-} "6";"7",
\ar@{-} "7";"8",
\ar@{-} "4";"9",
\end{xy}}

\newcommand{\TEseven}{%
\begin{xy} 0;<0.7cm,0cm>:
(0,0)*=0{\CoxBullet}="1" , 
(1,0)*=0{\CoxBullet}="2" , 
(2,0)*=0{\CoxBullet}="3" ,
(3,0)*=0{\CoxBullet}="4",
(4,0)*=0{\CoxBullet}="5",
(5,0)*=0{\CoxBullet}="6",
(6,0)*=0{\CoxBullet}="7",
(2,1)*=0{\CoxBullet}="8",
\ar@{-} "1";"2",
\ar@{-} "2";"3",
\ar@{-} "3";"4",
\ar@{-} "4";"5",
\ar@{-} "5";"6",
\ar@{-} "6";"7",
\ar@{-} "3";"8",
\end{xy}}

\newcommand{\TTEeightL}[1]{%
\begin{xy} 0;<0.7cm,0cm>:
(0,0)*=0{\CoxBullet}="1" , 
(1,0)*=0{\CoxBullet}="2" , 
(2,0)*=0{\CoxBullet}="3" ,
(3,0)*=0{\CoxBullet}="4",
(4,0)*=0{\CoxBullet}="5",
(5,0)*=0{\CoxBullet}="6",
(6,0)*=0{\CoxBullet}="7",
(7,0)*=0{\CoxBullet}="8",
(8,0)*=0{\CoxBullet}="9",
(2,1)*=0{\CoxBullet}="10",
\ar@{-} "1";"2",
\ar@{-} "2";"3",
\ar@{-} "3";"4",
\ar@{-} "4";"5",
\ar@{-} "5";"6",
\ar@{-} "6";"7",
\ar@{-} "7";"8",
\ar@{-}^{#1} "8";"9"
\ar@{-} "3";"10",
\end{xy}}

\newcommand{\TDfour}[1]
{\begin{xy} 0;<0.7cm,0cm>:
(0,0)*=0{\CoxBullet}="1" , 
(1,0)*=0{\CoxBullet}="2" , 
(0,1)*=0{\CoxBullet}="3" ,
(-1,0)*=0{\CoxBullet}="4",
(0,-1)*=0{\CoxBullet}="5",
\ar@{-}^{#1} "4";"1",
\ar@{-} "1";"3",
\ar@{-} "1";"2",
\ar@{-} "1";"5",
\end{xy}}

\newcommand{\TTDfour}[1]{%
\begin{xy} 0;<0.7cm,0cm>:
(0,0)*=0{\CoxBullet}="1" , 
(1,0)*=0{\CoxBullet}="2" , 
(2,0)*=0{\CoxBullet}="3" ,
(3,0)*=0{\CoxBullet}="4",
(2,-1)*=0{\CoxBullet}="5",
(2,1)*=0{\CoxBullet}="6",
\ar@{-}^{#1} "1";"2",
\ar@{-} "2";"3",
\ar@{-} "3";"4",
\ar@{-} "3";"5",
\ar@{-} "3";"6"
\end{xy}}

\newcommand{\TTDfive}{%
\begin{xy} 0;<0.7cm,0cm>:
(0,0)*=0{\CoxBullet}="1" , 
(1,0)*=0{\CoxBullet}="2" , 
(2,0)*=0{\CoxBullet}="3" ,
(3,0)*=0{\CoxBullet}="4",
(4,0)*=0{\CoxBullet}="5",
(1,1)*=0{\CoxBullet}="6",
(2,1)*=0{\CoxBullet}="7",
\ar@{-} "1";"2",
\ar@{-} "2";"3",
\ar@{-} "3";"4",
\ar@{-} "4";"5",
\ar@{-} "2";"6",
\ar@{-} "3";"7",
\end{xy}}

\newcommand{\TTDsix}{%
\begin{xy} 0;<0.7cm,0cm>:
(0,0)*=0{\CoxBullet}="1" , 
(1,0)*=0{\CoxBullet}="2" , 
(2,0)*=0{\CoxBullet}="3" ,
(3,0)*=0{\CoxBullet}="4",
(4,0)*=0{\CoxBullet}="5",
(5,0)*=0{\CoxBullet}="6",
(1,1)*=0{\CoxBullet}="7",
(3,1)*=0{\CoxBullet}="8",
\ar@{-} "1";"2",
\ar@{-} "2";"3",
\ar@{-} "3";"4",
\ar@{-} "4";"5",
\ar@{-} "5";"6",
\ar@{-} "2";"7",
\ar@{-} "4";"8",
\end{xy}}

\newcommand{\TTDseven}{%
\begin{xy} 0;<0.7cm,0cm>:
(0,0)*=0{\CoxBullet}="1" , 
(1,0)*=0{\CoxBullet}="2" , 
(2,0)*=0{\CoxBullet}="3" ,
(3,0)*=0{\CoxBullet}="4",
(4,0)*=0{\CoxBullet}="5",
(5,0)*=0{\CoxBullet}="6",
(6,0)*=0{\CoxBullet}="7",
(1,1)*=0{\CoxBullet}="8",
(4,1)*=0{\CoxBullet}="9",
\ar@{-} "1";"2",
\ar@{-} "2";"3",
\ar@{-} "3";"4",
\ar@{-} "4";"5",
\ar@{-} "5";"6",
\ar@{-} "6";"7",
\ar@{-} "2";"8",
\ar@{-} "5";"9",
\end{xy}}

\newcommand{\TTDeight}{%
\begin{xy} 0;<0.7cm,0cm>:
(0,0)*=0{\CoxBullet}="1" , 
(1,0)*=0{\CoxBullet}="2" , 
(2,0)*=0{\CoxBullet}="3" ,
(3,0)*=0{\CoxBullet}="4",
(4,0)*=0{\CoxBullet}="5",
(5,0)*=0{\CoxBullet}="6",
(6,0)*=0{\CoxBullet}="7",
(7,0)*=0{\CoxBullet}="8",
(1,1)*=0{\CoxBullet}="9",
(5,1)*=0{\CoxBullet}="10",
\ar@{-} "1";"2",
\ar@{-} "2";"3",
\ar@{-} "3";"4",
\ar@{-} "4";"5",
\ar@{-} "5";"6",
\ar@{-} "6";"7",
\ar@{-} "7";"8",
\ar@{-} "2";"9",
\ar@{-} "6";"10",
\end{xy}}

\newcommand{\PStar}
{\begin{xy} 0;<0.7cm,0cm>:
(0,0)*=0{\CoxBullet}="1" , 
(0,1)*=0{\CoxBullet}="2" , 
a(162)*=0{\CoxBullet}="3" ,
a(234)*=0{\CoxBullet}="4",
a(306)*=0{\CoxBullet}="5",
a(18)*=0{\CoxBullet}="6",
\ar@{-} "1";"2",
\ar@{-} "1";"3",
\ar@{-} "1";"4",
\ar@{-} "1";"5",
\ar@{-} "1";"6"
\end{xy}}

\newcommand{\SixLoop}[2]
{\begin{xy} 0;<0.7cm,0cm>:
(0,0)*=0{\CoxBullet}="1" , 
(0.866,0.5)*=0{\CoxBullet}="2" , 
(0.866,-0.5)*=0{\CoxBullet}="3" ,
(1.732,0.5)*=0{\CoxBullet}="4",
(1.732,-0.5)*=0{\CoxBullet}="5",
(2.732,0)*=0{\CoxBullet}="6",
\ar@{-} "1";"2",
\ar@{-}^{#2} "2";"4",
\ar@{-} "4";"6",
\ar@{-} "6";"5",
\ar@{-}_{#1} "5";"3",
\ar@{-} "3";"1",
\end{xy}}

\newcommand{\CQuad}[5]{%
\begin{xy} 0;<0.7cm,0cm>:
(0,0)*=0{\CoxBullet}="1" , 
(1,0)*=0{\CoxBullet}="2" , 
(1,1)*=0{\CoxBullet}="3" ,
(0,1)*=0{\CoxBullet}="4",
\ar@{-}^{#1} "1";"2",
\ar@{-}^{#2} "2";"3",
\ar@{-}^{#3} "3";"4",
\ar@{-}^{#4} "4";"1",
\ar@{-}|{#5} "1";"3"
\end{xy}}

\newcommand{\Diagrambox}[3]{\raisebox{0pt}[#1][#2]{#3}}

\newcommand{\CoxGrUIV}{%
  \begin{xy} 0;<0.7cm,0cm>:
    (0,0)*=0{\CoxBullet}="1" , 
    (0,1)*=0{\CoxBullet}="2" , 
    (-0.866,-0.5)*=0{\CoxBullet}="3" ,
    (0.866,-0.5)*=0{\CoxBullet}="4",
    \ar@{-}|{\infty} "1";"2",
    \ar@{-}|{\infty} "2";"3",
    \ar@{-}|{\infty} "2";"4",
    \ar@{-}|{\infty} "3";"4",
    \ar@{-}|{\infty} "1";"4",
    \ar@{-}|{\infty} "3";"1"
  \end{xy}
}

\newcommand{\CoxGrSItwo}[1]{\Lintwo{m}}

\newcommand{\CoxGrHCI}[2]{\LinT{#1}{#2}}

\newcommand{\CoxGrHNCLVI}{\Diagrambox{27pt}{6pt}{\TTEeightL{}}}